\documentclass[preprint,12pt]{elsarticle}
\usepackage{graphicx} 
\usepackage{amsmath}
\usepackage{amsthm}
\usepackage{amssymb}
\usepackage{fullpage}
\usepackage{xcolor}
\usepackage{array}
\usepackage{placeins}
\usepackage{hyperref}
\usepackage[capitalise]{cleveref}

\newcolumntype{?}{!{\vrule width 1pt}}

\newtheorem{thm}{Theorem}
\newtheorem{corr}{Corollary}
\newtheorem{lemma}{Lemma}

\title{Cross-Correlation Periodograms with Decaying Noise Floor for Power Spectral Density Estimation}
\author[1,*]{Mark Magsino}
\affiliation[1]{
organization = {United States Naval Academy}, 
addressline = {121 Blake Road}, 
city = {Annapolis}, 
state = {Maryland}, 
postcode = {21402}, 
country = {United States of America}}
\affiliation[*]{Corresponding Author: magsino@usna.edu}
\date{}

\begin{document}
\begin{abstract}
We present a statistical analysis of a variant of the periodogram method that forms power 
spectral density estimates by cross-correlating the discrete Fourier transforms of 
adjacent time windows. The proposed estimator is closely related to cross-power spectral 
methods and to a technique introduced by Nelson, which has been observed empirically to 
improve detection of sinusoidal components in noise.
We show that, under a white Gaussian noise model, the expected contribution of noise to 
the proposed estimator is zero and that the estimator is unbiased
under certain window alignment conditions.
This contrasts with classical estimators where averaging reduces variance but
not expected noise.
Moreover, we 
derive closed-form expressions for the variance and prove an upper bound on
the expected magnitude of 
the estimator that decreases as the number of windows increases. 
This establishes that the 
proposed method achieves a noise floor that decays with averaging, unlike standard 
nonparametric spectral estimators.
We further analyze the effect of taking the absolute value to enforce nonnegativity, 
providing bounds on the resulting bias, and show that this bias also decreases with the 
number of windows. Theoretical results are validated through numerical simulations.
We demonstrate the potential sensitivity to phase misalignment and 
methods of realignment. We also provide empirical evidence that the estimator is 
robust to other types of noise.
\end{abstract}

\maketitle

\section{Introduction}\label{intro}
In \cite{Nelson1993} Nelson proposed a variant of 
Bartlett's method \cite{Bartlett1950} 
that involved cross-correlating adjacent time windows to modify
the periodogram method
for enhanced channelized instantaneous 
frequency estimation. He gave empirical evidence that it is
an improvement over Bartlett's method in this specific context and
that it is a good estimator of that
frequency even in low SNR settings \cite{Umesh1996}. Nelson's method
has been utilized in this way for several
applications such as speech processing 
\cite{Alsteris2007,Lim2012,Nelson2001,Shafi2009} and acoustics
\cite{Auger2013,Fulop2006,Thakur2013}. Cross-correlation makes the estimator
complex and Nelson's original formulation computed the complex modulus 
after averaging cross-correlated time windows. In addition to the statistical
analysis of these types of estimators, we show that only using the real part
leads to greater noise suppression while largely preserving the spectrum of the
signal.

Since it is a variant of Bartlett's method of power spectral averaging, the proposed
method of Nelson falls under the broad category of nonparametric estimators
\cite{Stoica2005}.
As such,
one would expect that it would also be suitable
for estimating the power spectra of more general signals and not only for stationary tones with noise. Estimating power spectra of more
general signals is important in numerous applications \cite{Babadi2014}.
A non-exhaustive list of applications in this setting
includes EEG \cite{Petroff2016}, optics \cite{Elson1995}, and radar
\cite{Yoshikawa2021}.
There are numerous methods for power spectrum estimation, 
each of which carry distinct advantages and disadvantages over 
other methods.
Some methods for estimating power spectra include the aforementioned
Bartlett's method \cite{Bartlett1950}, Welch's method \cite{Welch1967}, 
Capon's method \cite{Capon2005}, 
autoregressive methods \cite{Box1976, Takalo2005}, 
or MUSIC \cite{Schmidt1986}.
Some modern approaches even include machine learning methods
\cite{Jennings2019, Romero2017, Thilina2013}.

The primary benefits of Bartlett's method for power spectrum estimation 
are its simplicity, its nonparametric nature, and the fact that 
averaging out noise is a natural technique in many applications. This is
because the Central Limit Theorem guarantees that the contribution of white Gaussian noise converges
to its expected power spectra. Moreover, as more windows are taken the variance of 
its contribution to the power spectrum decreases. However, 
this very same theorem also provides the main drawback.
The power spectral density is computed using the squared modulus of
discrete Fourier transform values. In particular, the expected value of the 
squared modulus of Gaussian noise DFT values is actually
the variance of the noise! This means that on average a fixed strictly positive
value is added to the power spectrum
of the intended signal regardless of the number of windows used.

The main idea for using cross-correlation of distinct
time windows is to force independence of the noise components. 
There are a few existing applications which use a similar cross-correlation idea 
that make this modification
of Bartlett's method compelling. For example, a common practice in seismology 
involves cross-correlating
long windows of pure noise recorded at two different stations to leverage
the fact that ambient noise at each station is independent \cite{Campillo2003}.

This technique is able
to reduce the noise floor after averaging, even if one takes the absolute value
to avoid the negative values that sometimes appear in cross-power spectral methods.
Moreover, this noise floor decreases as more
windows are used. This is very surprising as it is in sharp contrast to 
most nonparametric averaging methods. Methods such as Welch's method or
multitapering methods \cite{Thomson1982} can reduce the noise floor below
the actual variance of the noise, but these methods do not decrease the
noise floor any further as more windows are used. Nonetheless,
\cref{main thm} shows that cross-correlation indeed provides a reduction of the
noise floor and is able to reduce it further with more windows.

The aforementioned multitapering method is another method that leverages independent
noise realizations for spectrum estimation,
but there are a few key differences with the proposed idea.
Specifically, multitaper methods create uncorrelated spectral estimates by using a
set of orthogonal window functions on the data. This technique offers great
variance reduction but still does not completely eliminate the noise floor problem
that it shares with other nonparametric methods. Furthermore, multitaper methods are
generally averaged across orthogonal tapers, as opposed to distinct time windows
like Bartlett's method or Welch's method. 

Our main contribution is
the statistical analysis of the proposed Cross-Correlation Periodogram (CCP)
estimator,
which is similar to the cross power spectrum proposed by Douglas Nelson. 
While cross-spectral techniques are well known in signal processing for suppressing 
uncorrelated noise, the statistical behavior of cross-correlation periodograms formed 
between adjacent time windows has not been analyzed in detail. In particular, it is not 
immediately clear how averaging such cross-spectral estimates of the
same signal but in different time windows affects the expected 
contribution of noise and the variance of the resulting estimator.

The main result, \cref{main thm}, computes the exact first and second moments
of the estimator under a white Gaussian noise model. The efficacy of the
estimator is further highlighted by \cref{bias thm}, which shows that the CCP 
estimator is unbiased
for periodic signals windowed in alignment with its period. 
Moreover, we prove upper bounds on the 
expected value of the CCP estimator if the absolute value
is computed after the fact to eliminate negative power values.
In the case of white Gaussian noise 
the bound is significantly smaller than the variance of the noise when at least
three windows are used and further decreases with more windows.
\cref{with data exp} provides an upper bound on the bias when estimating
periodic signals with the absolute value of the CCP estimator.
These results provide a theoretical explanation for empirical observations reported in
the prior work of Douglas Nelson. We confirm these theoretical results through
numerical simulation. We also demonstrate examples that show the window
alignment assumption can be relaxed in applications and is primarily
an analysis tool more than an inhibitive restriction.

The rest of this work is organized as follows. In \cref{setup} we establish
the notation and setting for our analysis as well as the proposed
cross-correlation periodogram method under analysis. 
We also stress the necessity of the window alignment assumption for our 
statistical analysis but also highlight why it may not be prohibitive as it seems.
In \cref{exp value section}, we compute the first and second moments of the
cross-correlation periodogram estimator and prove a
bound on the absolute value of the estimator in the pure white Gaussian noise
setting. This is then extended to the window aligned 
periodic signals outlined in \cref{setup}. 
Specifically, the estimator is unbiased when applied
to a window aligned 
periodic signal. If the absolute value is computed to avoid negativity, 
we prove there is an error term independent of the number of 
windows used that is proportional to the magnitude of the DFT of the actual signal.
However, the incurred bias still decreases as the number of windows is increased.
In 
\cref{sim section},
we provide details on numerical simulations for empirical verification of the bounds
obtained. We highlight the effects that ignoring the window alignment
assumption may have, but demonstrate examples that show this condition can often
be relaxed and any signal annihilation can be fixed in practice. This emphasizes
the window alignment assumption as primarily an analysis tool and is not a prohibitive
restriction. Furthermore, we provide empirical evidence that the key properties
of the estimator persist on other kinds of non-Gaussian noise with the decaying
noise floor observed across Laplace, uniform, and autoregressive noise 
distributions.
In \cref{discussion}, we summarize the results and
discuss potential further directions of research for the proposed
method.

\section{Setup and Preliminaries}\label{setup}
\subsection{Model and Notation}\label{model}
We model the true signal as a set of discrete
measurements that repeat outside the measuring window.
A continuous signal would require integral interchanges and stochastic
process arguments for a rigorous treatment. The core idea is that
cross-correlating windows enforces independence across noise windows, reducing their
contribution.
Details about convergence and integral interchange that arises in continuous theory 
would only distract from the main arguments.
Although this leaves questions about convergence to continuous
signals, any implementation requires discretization
and sampling in practice. Modeling the true signal
as discrete finite samples is a more accurate reflection 
of how the estimator approximates the
discretely sampled noiseless version of the signal that one
would ideally obtain from measurements.

Suppose we have samples of a periodic signal with additive Gaussian white noise 
over $M$ periods and $L$ measurements per period. Our measurements can be
modeled with
\begin{equation}
    y(t) = x(t) + n(t), \hspace{1em} t = 0, \cdots, LM-1,
\end{equation}
where $x$ is the $L$-periodic signal, i.e., $x(L+t) = x(t)$ for all $t$, and $n(t)$
is a sequence of i.i.d.\@ Gaussian random variables with zero mean and variance $\sigma^2$.
The Gaussian noise assumption is standard in spectral estimation analysis and allows
us to appeal to Isserlis' Theorem in \cref{with isserlis}. This allows us to obtain
exact closed form expressions for the first and second moments 
rather than asymptotic approximations.

Let $X$ be the (normalized) Fourier transform of one period of $x$. That
is,
\begin{equation}
    X(f) = \frac{1}{\sqrt{L}}\sum_{t=0}^{L-1} x(t)e^{-2 \pi i f t / L}, 
    \hspace{1em} f = 0, \cdots, L-1.
\end{equation}
Let $n_m(t)$ denote the $m$-th observed period of the additive Gaussian white noise and
let $N_m(f)$ denote its Fourier transform. That is, $n_m(t) := n(mL+t)$ and
\begin{equation}
    N_m(f) = \frac{1}{\sqrt{L}} \sum_{t=0}^{L-1} n_m(t) e^{-2 \pi i f t / L},
\end{equation}
Using the same conventions as above, the Fourier transform of 
the $m$-th observed period of $y$ is then
\begin{equation}
    Y_m(f) = X(f) + N_m(f).
\end{equation}

We will often suppress the frequency variable $f$, e.g., write $Y_m(f)$ as $Y_m$. 
The frequency variable does not play a significant role in
proving the main results and is cumbersome to carry around. Similarly,
proofs which involve expanding Fourier transform sums often lead 
to many summations. When convenient, we express them as a single 
multi-indexed sum.

\subsection{Cross-Correlation Periodogram}\label{cps method}
Using the classic Bartlett method, the estimate of the power spectral density of $y$ is
\begin{equation}
P^{(B)}_y(f) = \frac{1}{M} \sum_{m=0}^{M-1} |Y_m(f)|^2,
\end{equation}
In contrast, we define the cross-correlation periodogram
(cf. \cite{Nelson1993}) by
\begin{equation}
P^{(CCP)}_y(f) =  \frac{1}{M} \sum_{m=0}^{M-1} \mathrm{Re}[Y_{m}(f)Y^*_{m+1}(f)],
\end{equation}
where window indices $m$ are taken modulo $M$, i.e., $Y_M = Y_0$.

The proposed estimator can also be viewed as the average of 
the real part of cross-power spectral
estimates between adjacent windows.
From this perspective, the proposed estimator may be further
interpreted as a lag-one cross-spectral averaging method applied to 
successive time windows.
These viewpoints are appealing since 
cross-spectral estimators are known to suppress uncorrelated noise because the 
expected cross-spectrum 
of independent noise realizations is zero. The analysis in 
\cref{exp value section}  provides a 
quantitative characterization of this phenomenon for the proposed estimator when applied 
to adjacent windows of the same signal.

One notable difference from the method proposed by Nelson is that we take
the real part instead of the absolute value.
One may wonder why we suggest such a modification.
The idea of cross-correlation methods is to look at the formula
$|Y_m|^2 = Y_mY_m^*$ and replace one term with an adjacent window. This
leverages the independence of the noise components across windows while
still having the same target power spectrum due to the periodicity assumption. However, 
$Y_{m}Y_{m+1}^*$ is not purely real and has an imaginary
component. Although this could be resolved by taking the modulus after
summing the terms (indeed that is what Nelson originally proposed), 
for the purposes of power spectral density estimation it turns out
to be advantageous to only take the
real part.

The following computation illustrates the intuition for why taking the
real part suppresses noise but maintains the target
signal. Writing
$Y_m = X+N_m$ once again, we have
\begin{equation}
Y_{m}Y_{m+1}^* = (X+N_{m})(X^*+N_{m+1}^*) 
= |X|^2 + XN_{m+1}^*+X^*N_{m}+N_{m}N_{m+1}^*.
\end{equation}
The component we ideally want, $|X|^2$, is purely real
while the remaining noise terms contain real and imaginary parts.
Thus, eliminating the imaginary part only removes unwanted noise
and provides even better reduction of noise than simply taking the
complex modulus after summing. 

It should be noted that it is possible for the proposed cross-correlation
periodogram estimator
to be negative. In such
cases this usually means the noise is greater than the contribution of that 
frequency component of the signal. One can compute the absolute value
for just those components and
this method will still produce a smaller value (i.e., less contribution from noise) 
compared to Bartlett's method or Nelson's original method.
Specifically, \cref{main thm} along with \cref{with data exp} give bounds
on the expected value of the absolute value version of the estimator
along with an upper bound on the bias.

\subsection{Relaxation of the Window Alignment Assumption}\label{needed periodicity}
Our model and subsequent statistical analysis relies on the assumption 
in \cref{model} that our desired signal, $x(t)$, was exactly $L$-periodic,
i.e., that each window ends an exact multiple of one period after the start. 
This ensures
that each successive window does not carry an additional phase shift,
which simplifies the statistical analysis.
However, the cross-correlation periodogram can still produce
good results even in the presence of phase shifts. For example,
suppose that windows sizes are chosen in a way that makes the signal
component of the $m$-th window a shifted copy of $x$ by $\tau_m$.
That is, the $m$-th window is given by
\begin{equation}
y_m(t) = x(t-\tau_m)+n_m(t).
\end{equation}
We want to focus on how phase affects the main signal portion, so for
simplicity for the time being, let us assume that $n_m(t) = 0$ for all $m$
and $t$.
Because of the shift in the time domain, the $m$-th Fourier transform
is now given by
\begin{equation}
X_m(f) = e^{2\pi i \tau_m f/L}X(f).
\end{equation}
Consequently, when we compute terms in the cross-correlation periodogram
we would take the real component of
\begin{equation}
X_mX^*_{m+1} = e^{2 \pi i (\tau_m-\tau_{m+1})f/L}|X(f)|^2.
\end{equation}
If we additionally assume that the phase shift between windows is a 
fixed constant, i.e., $\tau = \tau_{m}-\tau_{m+1}$ for all $m$, 
then the phase factor becomes
\begin{equation}
X_mX^*_{m+1} = e^{2 \pi i \tau f/L}|X(f)|^2.
\end{equation}
Once we compute the real part this becomes
\begin{equation}
\mathrm{Re}(X_mX^*_{m+1}) = \cos(2 \pi \tau f / L) |X(f)|^2. \label{phase issue}
\end{equation}
Thus, the issue mainly arises when $\cos(2 \pi \tau f /L)$ is close to zero.
If, for example, $\tau$ is exactly $L/4$ and $f \equiv 1 \mod 4$ then
the argument of the cosine function is $\pi/2$ resulting in a zero.

That being said, the experiments in \cref{benign phase effect} show that
the estimator remains robust with respect to several phase mismatches even without
realignment. Even when $\tau$ is relatively large compared to the target frequency,
the phase frequently ends up far away enough from the critical phases of
$\pi/2$ and $3\pi/2$ to identify signals even with slight penalties to the 
strength of the target signal.

Moreover, while \eqref{phase issue} may seem like a limitation at first glance,
the computation of that equation
actually shows us how to fix the issue. 
In this scenario since the phase misalignment
is known and constant, 
we can multiply by an offsetting phase factor, $e^{-2 \pi i \tau f/L}$,
before computing the real part to realign the phase. 
Even if the phase misalignment is not known, 
since the critical values are when the phase angle is close 
to $\pi/2$ or $3\pi/2$, one can compute two estimators: The proposed
cross-correlation periodogram as is and a modified one that multiplies
cross-correlations by $i$ before computing the real part and averaging. 
If the signal was
buried by being close to the critical phases, the modified estimator will
recover the signal by advancing it by $\pi/2$ in phase and pushing the
cosine factor back close to 1.
This allows us to keep the
highlighted advantages of taking the real part over only the complex 
modulus at the end, cf. \cite{Nelson1993}.
We demonstrate this zeroing of the primary signal and its realignment in
\cref{phase effect sim}. 

\section{Statistical Analysis of the Estimator}\label{exp value section}
\subsection{Analysis of the Estimator on White Noise}\label{noise analysis}
For the following, we will focus our analysis primarily on the effect of
the CCP estimator
on pure white Gaussian noise and compute the first and second moments
of $P^{(CCP)}_n$. 
Since the focus is on this estimator, we will omit the superscript and
write $P_n$ whenever no ambiguity arises.
We also suppress the frequency
variable $f$ for notational simplicity. Throughout this section, we will also implicitly
assume that $f \neq 0$ and $f \neq L/2$. This avoids the degenerate edge cases
without complicating the proof and analysis.

The main result is the computation of the first and second moments of the
estimator $P_n$ as well as an upper bound on the expected value of
$|P_n|$. In particular, the bound on the absolute value implies that
the bias of the estimator that arises when computing the absolute value
decays as the number of windows increases. This is further emphasized by
\cref{with data exp} which gives an upper bound on the bias when applied
to periodic signals.

\begin{thm}\label{main thm}
Let $n(t)$ be an i.i.d.\@ sequence of Gaussian random variables with mean zero
and variance $\sigma^2$,
divided into $M$ windows of length $L$. If $M \geq 3$ then,
\begin{itemize}
\item[a.]$\displaystyle\mathbb{E}(P_n) = 0,$
\item[b.]$\displaystyle\mathrm{Var}(P_n) = \frac{\sigma^4}{2M},$
\item[c.]$\displaystyle\mathbb{E}|P_n| \leq \frac{\sigma^2}{\sqrt{2M}}.$
\end{itemize}
\end{thm}

Before we prove this theorem, it is worth taking a moment to discuss
the meaning of this result. 
This result shows
that the expected contribution of noise to this estimator is zero.
This is key in proving \cref{bias thm}, which shows that
the cross-correlation periodogram is an unbiased estimator. 
Moreover,
even if we compute the absolute value of the noise to fix negative
values, i.e., $|P_n|$, the expected value still
tends to zero as the number
of windows increases. In contrast, for Bartlett's method,
increasing the number of windows does not lead to any improvement in
the expected contribution of noise. Indeed, under the assumptions of the
main theorem, $\mathbb{E}(P^{(B)}_n) = \sigma^2$ for any number of windows $M$ \cite{Bartlett1950}. 

The omission of $L$ in this equality is deliberate since the $L$ is canceled in
\eqref{L cancel start} -- \eqref{L cancel end} of the proof 
of \cref{Cm covariance}. This removes the dependence of the result on $L$.
In particular, this means that the benefits of this estimator are obtained regardless
of how finely (or coarsely) the signal is sampled. This is the reasoning for
the choice of the normalized discrete Fourier transform in \cref{model} and
the unnormalized version would add an additional factor of $L$ to the moments
and bounds obtained.

Note that even for just three windows, $\mathbb{E}|P_n|$ is
less than half of $\sigma^2$. This suggests that the cross-correlation
periodogram is an immediate improvement over Bartlett's method for reducing noise
and not just an asymptotic improvement that requires a large number of windows.
We provide empirical evidence that validate these bounds in \cref{sim section}. 

Before continuing, we note an immediate corollary of this theorem. The
variance of the cross-correlation periodogram estimator on noise
is strictly less than the variance obtained from Bartlett's method where
$\mathrm{Var}(P^{(B)}_n) = \sigma^4/M$.
Specifically, a variance of less than half of that of Bartlett's method is obtained,
additionally improving upon the main draw of Bartlett's method.

\begin{corr}\label{ccp var lemma}
Suppose $n(t)$ is an i.i.d.\@ sequence of Gaussian random variables
with mean zero and variance $\sigma^2$
divided into $M$ windows of length $L$. If $M \geq 3$, then
\begin{equation}
\mathrm{Var}(P^{(CCP)}_n) < \mathrm{Var}(P^{(B)}_n).
\end{equation}
\end{corr}

To break up proving \cref{main thm}, we
begin by detailing a few lemmas that help organize the computation of quantities
needed for the proof.
\begin{lemma}\label{with isserlis}
Let $M \geq 3$. For every choice of $t,t',s,s',m$ and $m'$, we have that
\begin{equation}
\mathbb{E}(n_m(t)n_{m+1}(t')n_{m'}(s)n_{m'+1}(s')) = 
\delta_{mm'}\delta_{ts}\delta_{t's'}\sigma^4.
\end{equation}
\end{lemma}

\begin{proof}
The argument of the expected value may be dependent for certain values of $t,t',s$ and $s'$. However, in those cases, the dependent terms are exact copies of each other, and so any linear combination of the terms is still normally distributed. Thus, the argument is jointly normal for every choice of parameters. By Isserlis' theorem we have
\begin{align}
\mathbb{E}(n_m(t)n_{m+1}(t')n_{m'}(s)n_{m'+1}(s')) &= \mathbb{E}(n_m(t)n_{m+1}(t')) \mathbb{E}(n_{m'}(s)n_{m'+1}(s')) \\
~& + \mathbb{E}(n_m(t)n_{m'}(s))\mathbb{E}(n_{m+1}(t')n_{m'+1}(s')) \\
~& + \mathbb{E}(n_m(t)n_{m'+1}(s'))\mathbb{E}(n_{m+1}(t')n_{m'}(s)).
\end{align}
Focusing on each term one at a time, both expectations in the first term are
expectations of products of random variables that are always independent. Hence
\begin{equation}
\mathbb{E}(n_m(t)n_{m+1}(t')) \mathbb{E}(n_{m'}(s)n_{m'+1}(s'))
= \mathbb{E}(n_m(t)) \mathbb{E}(n_{m+1}(t'))\mathbb{E}(n_{m'}(s))
\mathbb{E}(n_{m'+1}(s')) = 0,
\end{equation}
since each term has zero mean.

Turning our attention to the second term,
we have dependence only when windows and indices all match simultaneously. That is,
we only have dependence when $m=m'$, $s=t$, and $s'=t'$. In those cases
the arguments in each expectation are $n^2_m(t)$ and $n_{m+1}^2(t')$, respectively. Thus,
\begin{equation}
\mathbb{E}(n_m(t)n_{m'}(s))\mathbb{E}(n_{m+1}(t')n_{m'+1}(s'))
= \delta_{mm'}\delta_{ts}\delta_{t's'} \sigma^4.
\end{equation}

For the third and final term, a necessary condition for dependence would be $m = m'+1$ and
$m+1 = m'$. However, this would imply that $m' = m-1$ and $m'=m+1$ simultaneously, which
is only possible if $M = 2$. Since $M \geq 3$, the third term is also the product of two
expectations containing two independent zero mean variables, and is also zero.
Adding the results of all three cases together proves the lemma.
\end{proof}

\begin{lemma}\label{cosine sum}
Let $\theta_t = 2 \pi f t/L$ and $\theta_{t'} = 2 \pi f t'/L$.
Then,
\begin{equation}
\sum_{t,t'} \cos(2\theta_t-2\theta_{t'}) = 0. 
\end{equation}
\end{lemma}

\begin{proof}
First, note that we can write $\cos\theta$ in terms of $e^{i\theta}$ terms
using the identity
\begin{equation}
\cos\theta = \frac{e^{i\theta}+e^{-i\theta}}{2}.
\end{equation}
Hence,
\begin{equation}
\cos(2\theta_t - 2\theta_{t'}) = 
\frac{e^{2i\theta_t}e^{-2i\theta_{t'}} + e^{-2i\theta_t}e^{2i\theta_{t'}}}{2}.
\end{equation}
Notably, the second term is the complex conjugate of the first and it suffices
to show that the first term sums to zero. Since we expressed the terms
as complex roots of unity, we see that
\begin{equation}
    \sum_{t,t'} e^{2i\theta_t} e^{-2i \theta_{t'}} =
    \sum_t e^{2i \theta_t} \sum_{t'} e^{-2i \theta_{t'}}
     = 0,
\end{equation}
since each summation is now a sum of a full set of roots of unity.
\end{proof}

\begin{lemma}\label{Cm covariance}
Let $M \geq 3$. For each $m$, define the random variable $C_m:= \mathrm{Re}(N_mN^*_{m+1})$. Then,
\begin{itemize}
    \item[(a)] $\mathbb{E}(C_m) = 0$.
    \item[(b)] $\mathrm{Cov}(C_m,C_{m'}) = \delta_{mm'}\sigma^4/2$.
\end{itemize}
\end{lemma}

\begin{proof}
Let $A_m$ and $B_m$ be the real and imaginary parts of $N_m$, respectively. That is
$N_m = A_m + iB_m$. Expanding the product $N_mN_{m+1}^*$ in terms of these random
variables gives
\begin{align}
N_mN_{m+1}^* & 
= (A_m + iB_m)(A_{m+1}- iB_{m+1}) \\
~&= (A_m A_{m+1} + B_mB_{m+1}) + i(A_{m+1}B_m - A_{m}B_{m+1}).
\end{align}
Taking the real part, we obtain
\begin{equation}
C_m = A_mA_{m+1} + B_m B_{m+1}.
\end{equation}
Now, let $\theta_t: = 2 \pi f t/ L$ and $\theta_{t'} := 2 \pi f t' / L$. Since we can write $A_m$ and $B_m$ in terms of cosine and sine we have
\begin{equation}
A_mA_{m+1} = \frac{1}{L}\sum_{t,t'}\cos\theta_t\cos\theta_{t'}
n_m(t)n_{m+1}(t'),
\end{equation}
and
\begin{equation}
B_mB_{m+1} = \frac{1}{L}\sum_{t,t'}\sin\theta_{t}\sin\theta_{t'}
n_m(t)n_{m+1}(t').
\end{equation}
After adding $A_mA_{m+1}$ and $B_mB_{m+1}$, the $n_m(t)n_{m+1}(t')$ terms are factorable and
after the appropriate trigonometric sum identity we obtain
\begin{align}
C_m &= \frac{1}{L}\sum_{t,t'}(\cos\theta_t\cos\theta_{t'} + \sin\theta_t
\sin\theta_{t'}) n_m(t)n_{m+1}(t') \\
~ &= \frac{1}{L}\sum_{t,t'} \cos(\theta_t - \theta_{t'})n_m(t)n_{m+1}(t').
\end{align}

From here, $\mathbb{E}(C_m) = 0$ quickly follows since $\cos(\theta_t - \theta_{t'})$ is
deterministic. Using this and leveraging independence of separate windows we get
\begin{align}
\mathbb{E}(C_m) &=  \frac{1}{L}\sum_{t,t'} \cos(\theta_t - \theta_{t'})\mathbb{E}(n_m(t)n_{m+1}(t')) \\
~ &= \frac{1}{L}\sum_{t,t'} \cos(\theta_t - \theta_{t'})\mathbb{E}(n_m(t))\mathbb{E}(n_{m+1}(t'))  = 0.
\end{align}

To compute covariance, first note that because $\mathbb{E}(C_m) = 0$ for all $m$,
$\mathrm{Cov}(C_m,C_{m'}) = \mathbb{E}(C_mC_{m'})$. 
Expanding that quantity and applying \cref{with isserlis} gives
\begin{align}
    \mathbb{E}(C_mC_{m'}) &= \frac{1}{L^2}\sum_{t,t',s,s'} \cos(\theta_t-\theta_{t'})
    \cos(\theta_{s}-\theta_{s'})
    \mathbb{E}(n_m(t) n_{m+1}(t') n_{m'}(s) n_{m'+1}(s')) \label{L cancel start}\\
    ~ &= \frac{1}{L^2}\sum_{t,t',s,s'}
    \cos(\theta_t-\theta_{t'})
    \cos(\theta_{s}-\theta_{s'})
    \delta_{mm'}\delta_{ts}\delta_{t's'}\sigma^4 \\
    ~ &= \frac{\delta_{mm'}\sigma^4}{L^2}\sum_{t,t'}
    \cos^2(\theta_t-\theta_{t'})\\
    ~ &= \frac{\delta_{mm'}\sigma^4}{2L^2}\sum_{t,t'}
    (1+\cos(2(\theta_t-\theta_{t'}))) \\
    ~ &= \frac{\delta_{mm'}\sigma^4}{2} + \frac{\delta_{mm'}\sigma^4}{2L} \sum_{t'=0}^{L-1}
    \cos(2(\theta_t-\theta_{t'})) \\
    ~&= \frac{\delta_{mm'}\sigma^4}{2}, \label{L cancel end}
\end{align}
where the summation term is zero by \cref{cosine sum}.
\end{proof}

\begin{proof}[Proof of \cref{main thm}.] First, define $C_m := \mathrm{Re}(N_mN^*_{m+1})$ as in
\cref{Cm covariance}. Then, $P_n$ can be written as
\begin{equation}
P_n = \frac{1}{M}\sum_m C_m.
\end{equation}
From \cref{Cm covariance}, $\mathbb{E}(C_m) = 0$ for all $m$. Thus,
\begin{equation}
\mathbb{E}(P_n)=\mathbb{E}\left(\frac{1}{M} \sum_m C_m \right) = \frac{1}{M}\sum_m\mathbb{E}(C_m) = 0,
\end{equation}
establishing the first result.
Since $\mathrm{Cov}(C_m,C_m') = \delta_{mm'}\sigma^4/2$,
\begin{align}
    \mathrm{Var}(P_n) &=
    \frac{1}{M^2} \sum_{m,m'}\mathrm{Cov}(C_m,C_{m'}) \\
    ~ &= \frac{1}{M^2}\sum_{m,m'}\frac{\delta_{mm'}\sigma^4}{2} \\
    ~ &= \frac{\sigma^4}{2M},
\end{align}
establishing the second result.
We now immediately get that the second moment and variance are the same, 
i.e., $\mathbb{E}(P_n^2) = \mathrm{Var}(P_n)$.
Finally, by Lyapunov's inequality
\begin{equation}
\mathbb{E}|P_n| \leq [\mathbb{E}(P_n^2)]^{1/2} = \frac{\sigma^2}{\sqrt{2M}},
\end{equation}
which establishes the last result.
\end{proof}

\subsection{Analysis of the Estimator on Periodic Signals}\label{periodic analysis}
We now turn to analyzing the estimator when used on periodic signals.
As in the previous section, we omit the superscript of $P^{(CCP)}_y$ and write $P_y$
since there is no ambiguity with Bartlett's method in this section.
Applying \cref{main thm} to the model outlined in \cref{model}, we
see that $P_y$ is an unbiased estimator for periodic signals with
additive white Gaussian noise.

\begin{thm}\label{bias thm}
Suppose $y(t) = x(t) + n(t)$ is divided into $M$ windows of length $L$ where
$x(t)$ is $L$-periodic and $n(t)$ is a length $LM$ sequence of i.i.d.\@ Gaussian random
variables with mean zero and variance $\sigma^2$. If $M \geq 3$, then
\begin{equation}
\mathbb{E}(P_y) - |X|^2 = 0.
\end{equation}
\end{thm}

\begin{proof}
Recall that $Y_mY^*_{m+1} = |X|^2 + XN_{m+1}^*+X^*N_{m}+N_{m}N_{m+1}^*$.
Furthermore, since $|X|^2$
is purely real and deterministic we have that
\begin{align}
    \mathbb{E}(P_y) - |X|^2 &= 
    \mathbb{E}\left(\frac{1}{M}\sum_{m=0}^{M-1}\mathrm{Re}
    (XN^*_{m+1} + X^*N_m+N_mN^*_{m+1})\right)\\
    ~&= \frac{1}{M}\sum_{m=0}^{M-1} \mathrm{Re}(X \mathbb{E}(N^*_{m+1}) + X^*
    \mathbb{E}(N_m))  + \mathbb{E}
    \left(\frac{1}{M} \sum_{m=0}^{M-1}\mathrm{Re}
    (N_mN^*_{m+1})\right) \\
    ~&= \frac{1}{M}\sum_{m=0}^{M-1} \mathrm{Re}(X \mathbb{E}(N^*_{m+1}) + X^*
    \mathbb{E}(N_m))  + \mathbb{E}(P_n). \label{bias computation}
\end{align}
Since $N_m$ and $N^*_{m+1}$ are sums of zero mean independent Gaussian random variables,
their expectation is zero. The last term has expectation zero by \cref{main thm}
and thus the total expectation is zero.
\end{proof}

As previously noted, since power spectral density is supposed to be positive,
we can instead compute $|P_y|$ to ensure positivity at
the cost of some bias. This is in line with cross-power spectrum
techniques where the absolute value is used at a similar cost of bias \cite{CNelson2014}.
We can compute an upper bound on the bias using the
last part of \cref{main thm}. This confirms the property that the bias
decreases as the number of windows $M$ increases and gives an asymptotic bound.

\begin{thm}\label{with data exp}
Suppose $y(t) = x(t) + n(t)$ is divided into $M$ windows of length $L$ where $x(t)$ is
$L$-periodic and $n(t)$ is a length $LM$ sequence of i.i.d.\@ Gaussian random variables
with mean zero and variance $\sigma^2$. If $M \geq 3$ then,
\begin{equation}
\mathbb{E}|P_y(f)|-|X(f)|^2\leq 
\sigma \sqrt{\pi}|X(f)| +\frac{\sigma^2}{\sqrt{2M}}.
\end{equation}
\end{thm}

Although we still omit the frequency variable $f$ in the proof, 
in this case we include the frequency index $f$ in the theorem statement.
In general $|X(f)|^2$ will not have
a uniform spectrum and so we include the variable to stress this dependence.
In particular, when $|X(f)|^2$ is close to zero, the bias is essentially
only coming from the noise.

\begin{proof}
Since $|X|^2$ is deterministic, note that
\begin{equation}
    \mathbb{E}|P_y| - |X|^2 = \mathbb{E}(|P_y|-|X|^2),
\end{equation}
and so we analyze the expected value of $|P_y| - |X|^2$.
If we expand the argument of the complex modulus in $P_y$,
the $|X|^2$ term is independent of the index $m$ and we can write
\begin{equation}
|P_y| = \left|
|X|^2 + \frac{1}{M}\sum_{m=1}^M \mathrm{Re}(XN_{m+1}^*+X^*N_m+N_mN_{m+1}^*)
\right|.
\end{equation}
Using the triangle inequality to separate the $|X|^2$ term, we have that
\begin{equation}
    |P_y| \leq |X|^2 + \frac{1}{M} \left|
    \sum_{m=1}^M \mathrm{Re}(XN^*_{m+1} + X^*N_m + N_m N^*_{m+1}) 
    \right|. \label{triangle Py}
\end{equation}
Thus, subtracting $|X|^2$ from both sides of \eqref{triangle Py} and
applying the triangle inequality two more times we get
\begin{align}
    |P_y| - |X|^2 &\leq \frac{1}{M}\left|
        \sum_{m=1}^M
        \mathrm{Re}(XN_{m+1}^*+X^*N_m+N_mN_{m+1}^*)
    \right| \\
    ~ &\leq \frac{1}{M} \left| \sum_{m=1}^M \mathrm{Re}(XN_{m+1}^*
    + X^*N_m) \right| + \left| \frac{1}{M}\sum_{m=1}^M 
    \mathrm{Re}(N_mN^*_{m+1}) \right| \\
    ~ &\leq |P_n| + \frac{1}{M} \sum_{m=1}^M |XN_{m+1}^*| + |X^*N_m|.
    \label{bias ineq computation}
\end{align}

Now we examine the expected value of each component of the resulting right
hand side sum.
For every frequency index $f$,
an elementary computation shows that
$N_m(f)$ is equivalent to a random vector in $\mathbb{R}^2$ where each component
is a Gaussian random variable with mean 0 and variance
$\sigma^2/2$. This comes from the fact that
$N_m(f)$ is obtained from a sum of independent complex Gaussian random variables.
In particular, this viewpoint means that $|N_{m+1}|$ is a Rayleigh distribution
with expected value $\mathbb{E}|N_{m+1}| = \frac{\sigma\sqrt{\pi}}{2}$.
Since $X$ 
itself is not a random variable the expected value of the first term in
the summation is
\begin{equation}
\mathbb{E}|XN^*_{m+1}| = |X|\cdot\mathbb{E}|N^*_{m+1}| = 
\frac{\sigma\sqrt{\pi}}{2} |X|. \label{bias signal noise 1}
\end{equation}
Similarly, the second term of the summation gives
\begin{equation}
\mathbb{E}|X^*N_m| = |X|\cdot\mathbb{E}|N_m| = 
\frac{\sigma\sqrt{\pi}}{2} |X|. \label{bias signal noise 2}
\end{equation}
Since those terms are not dependent on $m$ we see that the expected value
of $|P_y| - |X|^2$ is bounded by
\begin{equation}
\mathbb{E}(|P_y| - |X|^2) \leq
\mathbb{E}|P_n| + \frac{1}{M}\sum_{m=1}^M \mathbb{E}(|XN_{m+1}^*|
    + |X^*N_m|) \leq \sigma \sqrt{\pi}|X|+\frac{\sigma^2}{\sqrt{2M}},
\end{equation}
which establishes the result.
\end{proof}

Note that the statements of \cref{bias thm} and \cref{with data exp} rely on
the first moment and the bound on the absolute value of first moment that was
computed in \cref{main thm} for Gaussian noise. However, the actual proof
technique holds for any type of noise as the values associated with the noise
profile are only used at the end of the computation. That leads to the following
corollary of \cref{bias thm} and \cref{with data exp}

\begin{corr}\label{non-gauss corr}
Suppose $y(t) = x(t) + n(t)$ is divided into $M$ windows of length $L$ where $x(t)$ is
$L$-periodic and $n(t)$ is a length $LM$ sequence of i.i.d.\@ random variables with
zero mean. Then, for any choice of $m$,
\[
\mathbb{E}(P_y) - |X|^2 = 0
\]
and
\[
\mathbb{E}|P_y(f)| - |X(f)|^2 \leq 2|X(f)|\,\mathbb{E}|N_m(f)| + \mathbb{E}|P_n(f)|.
\]
\end{corr}

\begin{proof}
Fix a frequency $f \neq 0, L/2$. Since $n(t)$ is independent for all $t$, 
the Fourier transform of the $m$-th window is independent
of any other window. In particular, that means the Fourier transform
of the $m$-th window at frequency $f$, $N_m(f)$, is also i.i.d.\@ as a sequence
of $M$ random variables indexed by $m$ for this fixed $f$. 
Thus, $\mathbb{E}|N_m(f)|$ is the
same for any arbitrary choice of $m$.

Returning
to \eqref{bias computation}, all of the same computations still hold
and moreover, the proof of the first claim in \cref{Cm covariance}
did not use any specific facts about the Gaussian distribution. Specifically,
the property that allowed us to conclude 
$\mathbb{E}(\mathrm{Re}(N_mN^*_{m+1})) = 0$ was only the fact that
the expected value of the noise at any time $t$ was zero, regardless of
window. For the cross terms, since $N_m$ and $N^*_{m+1}$ are complex
linear combinations of random variables with zero mean, the expected
value of the real part of each is zero as well. This means the entire
right hand side of \eqref{bias computation} is zero which
establishes the first result.

Similarly, returning to \eqref{bias ineq computation}, we see that all
of the same computations hold except in the last step, where we do not 
know the values of $\mathbb{E}|N_m(f)|$ and $\mathbb{E}|P_n|$. However, the first
equalities in both
\eqref{bias signal noise 1} and \eqref{bias signal noise 2} are still valid
and so replacing them in the quantity obtained in \eqref{bias ineq computation}
gives the second result.
\end{proof}

The assumption that the noise has zero mean is both practical and useful.
If the noise has zero mean, this amounts to a constant value being added
to the DC component of the signal, and can thus be absorbed into the 
original signal $x(t)$. This affects the $f=0$ of the spectrum, which we
have explicitly ignored in this analysis.
In general, extending past the Gaussian case is non-trivial. Computing
the covariances of $\mathrm{Re}(N_m N^*_{m+1})$ between different values
of $m$ required the use of Isserlis' Theorem in \cref{with isserlis}.
The generalization of this, the moment-cumulants formula 
\cite{Leonov1959}, would require the fourth cumulant of the distribution
associated with the noise. This is not always easy to compute and may
not exist for some distributions. For example, the
Cauchy distribution has no finite moments and thus has no fourth cumulant.

\section{Simulations and Results}\label{sim section}
In this section we describe several numerical simulations and experiments
to illustrate and verify the theoretical bounds obtained in \cref{exp value section}.
For transparency and replicability, all code can be found on
Zenodo at \href{https://doi.org/10.5281/zenodo.17526546}{https://doi.org/10.5281/zenodo.17526546}.

\subsection{Direct Comparison to Bartlett and Welch Methods}\label{comparison}
\begin{figure*}[h]
\centering
\includegraphics[width=3.6in]{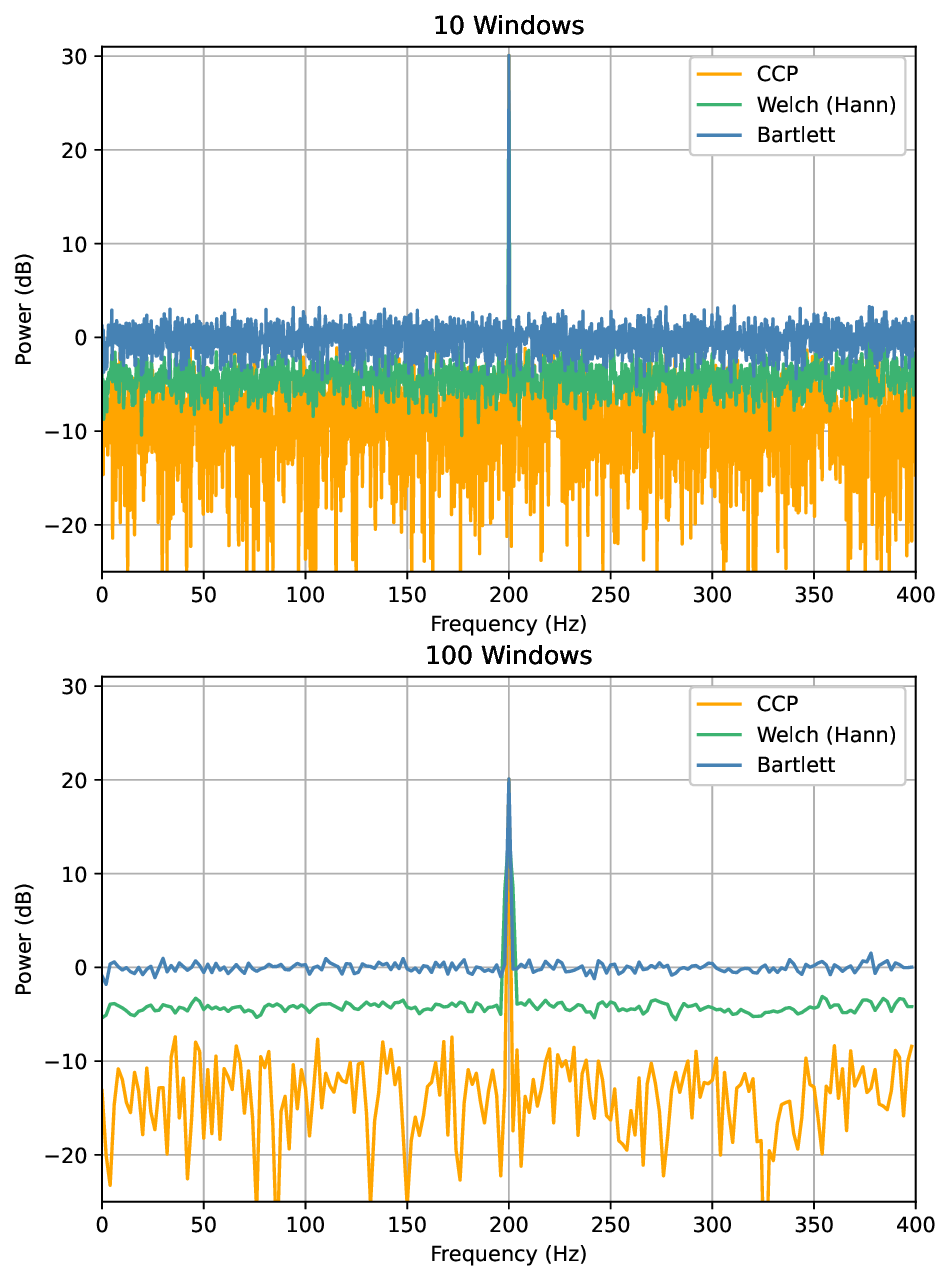}
\caption{PSD estimation using the proposed CCP method, Bartlett's method, and Welch's method (Hann window, no overlap) on a 200 Hz signal with additive Gaussian
white noise ($\sigma^2 = 1)$. Top subplot uses 10 windows and bottom uses 100 windows.
Going from 10 to 100 windows does not decrease
the 0 dB noise floor in Bartlett's method or the -4 dB noise floor in 
Welch's method. In contrast, the CCP method
achieves a noise floor decreasing from -8 dB (10 windows) to -13 dB (100 windows). 
Notably,
both are lower than the noise floor from Bartlett's and Welch's method.
}
\label{graph comp}
\end{figure*}

For this experiment we added white Gaussian noise
to a 200 Hz cosine wave and sampled at a rate of 800 Hz. The noise is
set to have variance 1. We then divide the signal into a set of 
10 non-overlapping windows and a set of 100 non-overlapping windows.
The same signal is used for both the 10 and 100 window methods to ensure
as direct of a comparison as possible. The proposed cross-correlation periodogram method,
Bartlett's method, and Welch's method with a Hann window and no overlap are applied to 
both sets. 
To avoid negativity, the absolute value of the cross-correlation periodogram 
is computed.
We then compare the power spectral density estimates all three 
methods produce. To ensure proper comparison and replicability of the results,
Bartlett and Welch methods are manually implemented. 
The results are depicted in \cref{graph comp}.

For all three methods we see the variance of the PSD estimate 
decrease when the number of windows
is increased from 10 to 100. However for Bartlett's method, we see the
well-known phenomenon where the noise floor does not decrease when going
from $M=10$ to $M=100$ windows and sits at $0$ dB for both estimates. Similarly, for Welch's method we see that same phenomenon where
the noise floor does not decrease when going from $M=10$ to $M=100$ windows and
sits at $-4$ dB in both estimates. 

In contrast, we see a decrease
in noise floor from $-8$ dB to
$-13$ dB in the cross-correlation periodogram method as we go from $M=10$ to $M=100$
windows. 
Note that these are lower than
the bounds predicted by \cref{main thm} which would imply the
noise floor should be at most $-6.5$ dB and $-11.5$ dB for
$M=10$ and $M=100$ windows, respectively. This is unsurprising as
the bound is really obtained by applying Lyapunov's inequality to the second moment
but equality is only achieved when the random variable
is constant almost surely.
\FloatBarrier

\subsection{Validation of Bounds}\label{bound validation}
\begin{figure*}
\centering
\begin{tabular}{|c|c|c|c|c|}
\hline
\# Windows & Mean $\pm$ SE & Mean Bound & Sample Mean Sq & Pred Mean Sq \\
\hline
3 & $0.2955 \pm 0.0027$ & 0.4082 & 0.1616 & 0.1667\\
5 & $0.2358 \pm 0.0022$ & 0.3162 & 0.1028 & 0.1000\\
10 & $0.1723 \pm 0.0014$ & 0.2236 & 0.0506 & 0.0500\\
25 & $0.1110 \pm 0.0009$ & 0.1414 & 0.0201 & 0.0200\\
100 & $0.0558 \pm 0.0004$ & 0.0707 & 0.0049 & 0.0050\\
1000 & $0.0179 \pm 0.0001$ & 0.0224 & 0.0005 & 0.0005\\
\hline
\end{tabular}
\caption{Table of empirical expected values compared with the predicted bound
from Theorem 1. 100 samples are taken per interval from $N(0,1)$. Empirical expected
value is the mean of the 10th frequency component over 10,000 trials.}
\label{bound comparison table}
\end{figure*}

In \cref{main thm} we obtain bounds on the expected value and variance
of the cross-correlation periodogram estimator when used on 
white Gaussian noise.
These bounds did not depend on the number of samples drawn per window.
For this experiment, we generate samples of white Gaussian noise, draw 100 samples per window, and average the 10th frequency bin
over $M$ windows. Since in expected value the spectrum of white noise
is uniform throughout, the choice of 10th frequency is arbitrary.
We perform this 10,000 times for each value of $M = 3,5,10,25,100,$ and
$1000$. This effectively creates 10,000 samples drawn from the distribution
of CCP estimates on white Gaussian noise at each parameter level $M$.

In \cref{bound comparison table} we note the empirical mean with
standard error for each $M$ and compare it to the bound predicted
by \cref{main thm}. We also compute the sample second moment and compare
it to the predicted second moment also predicted in \cref{main thm}.
We see the mean decrease as predicted and the empirical means are consistent
with the predicted bound. In fact, there appears to be a significant
gap between the predicted and empirical bounds suggesting that the
bounds computed are not sharp. Nonetheless, the empirical evidence
sufficiently establishes the main advantage of the 
cross-correlation periodogram
method over Bartlett's method. That is, the expected contribution
of noise does indeed decay as the number of windows increases while
retaining the property of Bartlett's estimator of reducing the
variance of the estimator as the number of windows increases.

\subsection{Phase Effects with Mild Misalignment}\label{benign phase effect}
\begin{figure}[h]
\centering
\includegraphics[width=4.5in]{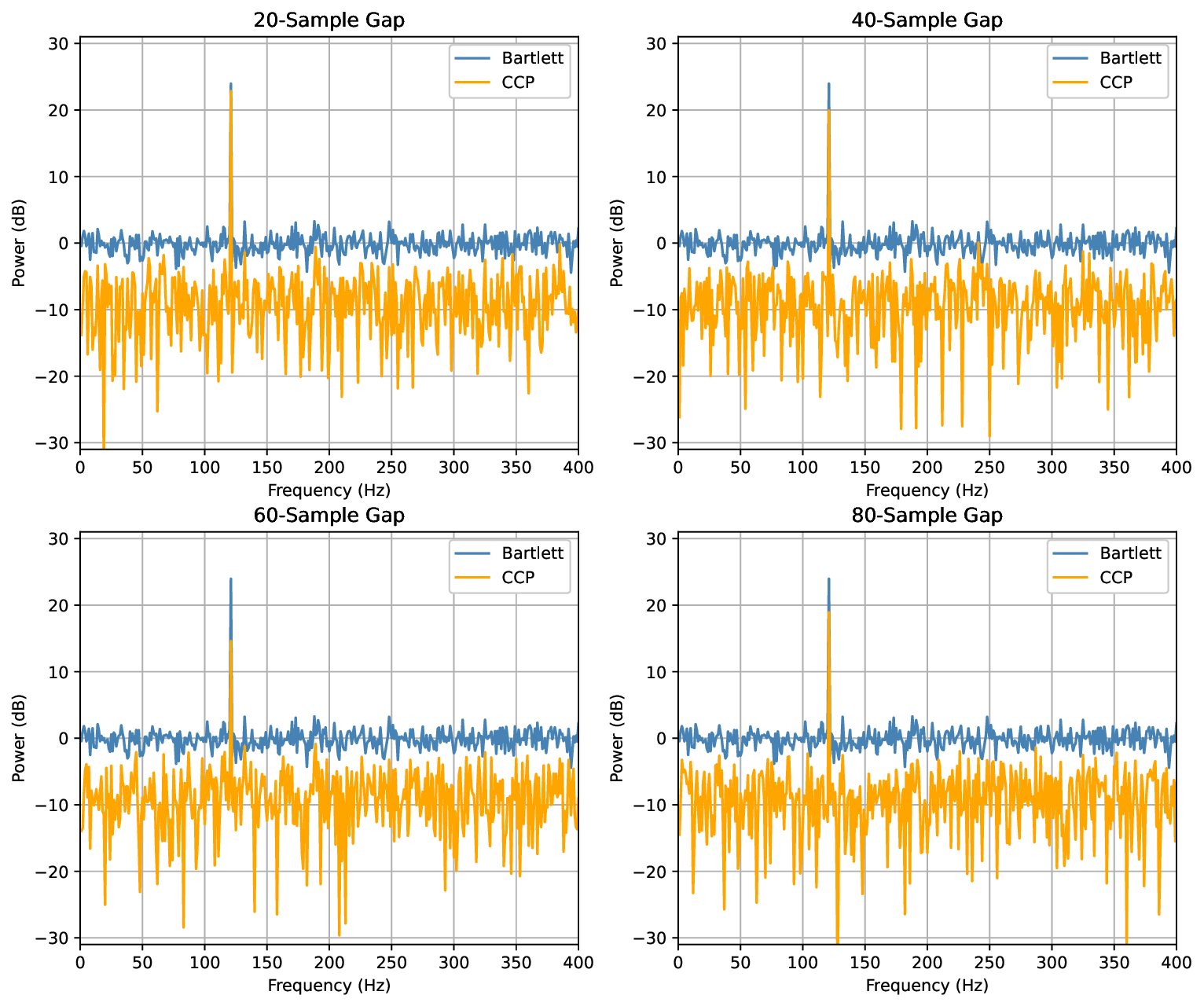}
\caption{121 Hz sinusoidal with additive white Gaussian 
noise divided into 10 windows of length 1 s. Gaps of 20, 40, 60, and
80 samples are introduced between windows. Because of the gaps, a slight
phase mismatch reduces the 121 Hz signal slightly but enough of the signal
is preserved to detect the frequency in all four cases.}
\label{mild phase effects}
\end{figure}

In this experiment we show that even if the windows are not
aligned to start and end at the exact same phases the CCP
estimator is still able to reduce noise and preserve the intended
signal reasonably well. We generate a pure sinusoidal signal with frequency
$f_0 = 121$ Hz sampled at 1000 Hz. For this experiment, each window has length
one so $f_0 = 121$ is the correct index for the 121 Hz frequency component. 
We then add white Gaussian
noise to the signal. The signal is split into 1 s windows but with a gap
between each window. This experiment is repeated four times with a gap of 20, 40,
60 and 80 samples. Recall from \cref{needed periodicity}
that the resulting phase factor is $\exp(2 \pi i f_0 \tau/L)$ for each
gap $\tau$.
Since $L=1000$ this results in phase shifts of 0.42, 0.84, 0.26, and 0.68 periods
of the 121 Hz signal, respectively. 

In all four cases, the signals are preserved enough to be detected. In the 20 sample
gap case, the signal is nearly perfectly preserved. The worst performing of these 
is the 60 sample gap, but even in that scenario the signal is at least preserved
well enough to be detectable. The 60 sample gap performs the worst because 
0.26 periods is the closest to the catastrophic scenario detailed
in \cref{phase effects} where the phase shift is 0.25 periods. This shows that even
when close to the critical numbers of 0.25 or 0.75 periods (phase factors of $i$ and
$-i$), the estimator can still
preserve signals well enough to be detectable. On the other hand, the 20 sample
gap performs best because it is the closest to 0 or 0.5 periods where the phase
factors would be 1 and -1 where the true signal would be fully kept in the real part.
\subsection{Misalignment Annihilation and Recovery via Realignment}
\label{phase effect sim}

\begin{figure}[h]
\centering
\includegraphics[width=3.5in]{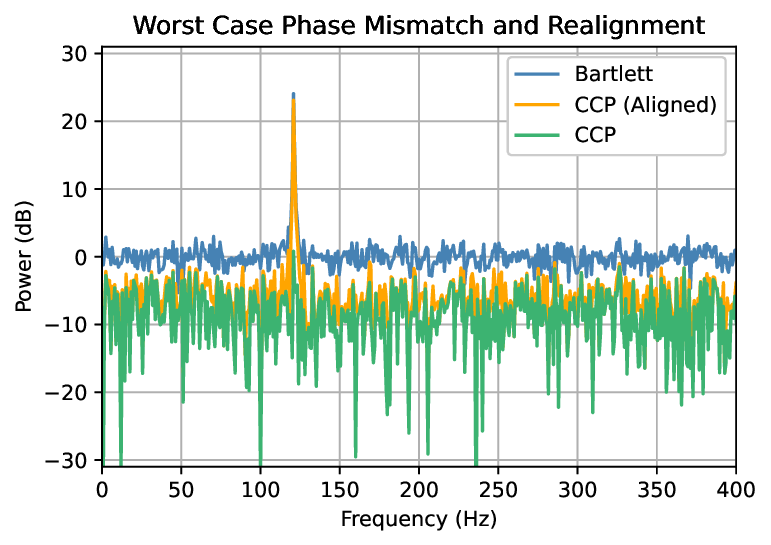}
\caption{121 Hz sinusoidal with additive white Gaussian 
noise divided into 10 windows of length 1.25 s. Because of
the mismatched window length and the target frequency the
cross-correlation periodogram (CCP) method completely eliminates the 121 Hz signal
in its PSD estimation while Bartlett's method does not
encounter this issue. Despite this, multiplying cross-correlations by $i$
before computing the real part restores the signal.}
\label{phase effects}
\end{figure}

In this experiment we show the phase effects that window shifts
can have on this method. To show this,
we generate a pure sinusoidal signal with frequency $f_0 = 121$ Hz and
sampled at 1000 Hz. This signal also has
additive white Gaussian noise.
The signal is then divided into non-overlapping windows of length 1.25 s to
simulate a mismatch that may arise from not knowing what the expected frequencies
of the target in advance.
This is effectively a gap of 250 samples between windows, i.e., a shift of $\tau = 250$ which
is exactly 1/4th of the sample rate. Since $121 \equiv 1 \mod 4$, 
this means that the
Fourier transform at $f_0$ of each window $Y_{k}(f_0)$ differs from the previous
window $Y_{k-1}(f_0)$ by
\begin{equation}
    Y_k(f_0) = e^{\pi i /2} Y_{k-1}(f_0) = iY_{k-1}(f_0).
\end{equation}
In particular, when looking at the estimator given by the cross-correlation
periodogram method, this means that
\begin{equation}
P_{y}(f_0) = \mathrm{Re}(i|X(f_0)|^2 + \text{ noise terms }),
\end{equation}
which completely kills off the intended target. 
However, by undoing the shift by multiplying by $-\exp(\pi/2) = -i$ before
computing the real part, the signal is restored.

These observations can be seen in the results depicted in
\cref{phase effects}.
Note that the same reduction in noise
floor and variance promised by \cref{main thm} is still obtained even
when the primary signal is eliminated completely. This is consistent
with \cref{main thm} as the theorem does not involve the periodicity condition
and makes no guarantees on the signal if
it is not properly periodically sampled or adjusted.

\subsection{Non-Gaussian Noise}\label{non-gauss experiment}
\FloatBarrier

\begin{figure}
\centering
\begin{tabular}{|c||c|c|c|c|c|}
\hline
~ & \multicolumn{5}{c|}{Number of Windows}\\
\hline
Noise Type & 3 & 5 & 10 & 25 & 100 \\
\hline
Laplace & 0.3051 & 0.2331 & 0.1764 & 0.1105 & 0.0560 \\
Uniform & 0.2954 & 0.2347 & 0.1671 & 0.1138 & 0.0566 \\
AR(1) & 0.3988 & 0.3253 & 0.2348 & 0.1463 & 0.0721 \\
\hline
\end{tabular}
\caption{Table of empirical expected noise contribution
when the cross-correlation periodogram is used on
non-Gaussian noise. All three exhibit the same noise 
contribution decay
proportional to the square root of the number of windows.}
\label{non-gaussian table}
\end{figure}

The main purpose of using Gaussian noise was to appeal to
Isserlis' theorem, as it is the simplest case of the
moment-cumulant expansion of the product of four
random variables. However, since $N_m$ itself is a linear combination
of $L$ noise variables, the Central Limit Theorem will guarantee
that the Fourier transform of the noise windows will be
nearly Gaussian in many cases if each window has enough samples. We show this by testing
the estimator on Laplace, uniform, and order-one
autoregressive noise. 

We generate samples of each type of noise with 100 samples
per window and average the 10th frequency bin over
$M$ windows. For Laplace and uniform noise, the variables
are independent and identically distributed and normalized
to have mean zero and standard deviation 1. For the
autoregressive noise, we use an autoregressive coefficient
of $0.5$ and use innovation variance $0.75 = 1-0.5^2$ so that the variance
of the total autoregressive noise is $1$
for all times $t$.

We do this process 1,000 times for each value of 
$M = 3, 5, 10, 25$ and $100$. As with \cref{bound validation},
this effectively creates 1,000 samples drawn from 
the distribution
of CCP estimates on each type of noise. For the Laplace and uniform noise,
the i.i.d.\@ assumption means the spectrum is flat and there 
is no difference between any frequency bins. For the
autoregressive noise, different bins will have different
scales because of the non-flat spectrum, but the same
decay behavior suggested by \cref{main thm} is still seen despite the
arbitrary choice of frequency.
Even though autoregressive noise is not independent across time samples within a window,
it is still independent across windows. The driving force behind the CCP estimator's
noise reduction is independence across windows, which explains the decaying
contribution of noise even for colored noise.
In particular, all three exhibit the same decay in noise
contribution proportional to $1/\sqrt{M}$ as suggested
by \cref{main thm}.

\section{Discussion and Future Work}\label{discussion}
The main consequence of \cref{main thm} is that the Cross-Correlation
Periodogram (CCP) method is an estimator that has zero contribution from noise in
expected value. Furthermore, even if absolute values are ultimately computed to
avoid negativity, the estimator still results in
reduced contribution from noise that also 
decays as the number of windows increases. 
This is in contrast to Bartlett's method where more windows
only decreases variance but not the expected value.
\cref{bias thm} shows that this allows for an unbiased estimator or a
biased estimator whose bias decays as the number of windows grows.
\cref{non-gauss corr} takes this a step further and shows that the CCP is an
unbiased estimator with any zero mean i.i.d.\@ additive noise. 
These results are empirically verified by the results of \cref{comparison}.
Moreover, the bounds obtained by \cref{main thm}
are further supported by empirical evidence in \cref{bound validation} and
\cref{non-gauss experiment} suggests empirically that the bounds essentially
still hold for non-Gaussian noise.

The statistical analysis of the cross-correlation periodogram estimator
relied on a window alignment assumption with respect to the signal
samples. This is because as noted in \cref{needed periodicity}, 
phase-shifted windows may cause effects that can even
potentially nullify the contribution of a desired signal entirely.
Techniques for handling such a phase shift are demonstrated, but a more detailed analysis
of those corrective techniques and how they can be used
to allow for overlapping or tapered windows is left to future work. 
This would effectively extend the cross-correlation periodogram method
in a similar fashion to how Welch's method \cite{Welch1967} extends Bartlett's method.

From the statistical lens, the second moment and subsequent bound in \cref{main thm}
is obtained in part through Isserlis' theorem in \cref{with isserlis}. As noted
in \cref{periodic analysis}, proper generalization would require moment-cumulant formulas
\cite{Leonov1959}, but may be tractable for other kinds of non-Gaussian noise.
Furthermore, the bound from \cref{main thm} is obtained by using Lyapunov's inequality
on the second moment. This allows for a simple bound for the first moment leveraging
the fact that we can compute the second moment exactly. However, as noted in 
\cref{comparison}, this is a loose bound since Lyapunov's inequality achieves equality
if and only if the random variable in question is constant almost surely. 
It is possible tighter bounds could be obtained by examining higher moments, e.g.,
the fourth moment. 

\section*{Acknowledgments}
The views expressed in the paper are those of the author and do not reflect the 
official policy or position of the U.S. Naval Academy, 
Department of the Navy, or 
the U.S. Government. 
The author would like to thank Trey Morris, Celeste Brown, and Garret Vo
of the Naval Research Lab for informing him about the prior works of
Douglas Nelson. The author would also like to thank fellow Mathematics Department
colleagues Junhyung Park and Douglas VanDerwerken for their helpful suggestions
and expertise in
statistics and probability theory.

\bibliography{citations}
\bibliographystyle{elsarticle-num}
\end{document}